\documentclass[final]{siamart0516}
\usepackage{amsfonts}
\usepackage{latexsym}
\usepackage{amsmath}
\usepackage{subfigure}
\usepackage{amssymb}
\usepackage{array}
\usepackage{float,color}
\usepackage[]{hyperref} %% [colorlinks]
\usepackage{fancyhdr}
\usepackage{epsfig}

\newtheorem{remark}[theorem]{Remark}
\newtheorem{assumption}[theorem]{Assumption}

\newcommand{\norm}[1]{\left\| #1 \right\|}
\newcommand{\R}{\mathbb{R}}

\newcommand{\ii}{\mathrm{i}}

\newcommand{\dd}{\mathrm{d}}

\newcommand{\tu}{\tilde{u}}

\title{Inverse boundary value problem for the Helmholtz equation:
       quantitative conditional Lipschitz stability estimates
       \thanks{This research was supported in part by the members, BGP, ExxonMobil,
PGS, Statoil and Total, of the Geo-Mathematical Imaging Group now at
Rice University.}}
\author{
Elena Beretta\thanks{Dipartimento di Matematica ``Brioschi'',
  Politecnico di Milano, Italy
  (\href{mailto:elena.beretta@polimi.it}{\texttt{elena.beretta@polimi.it}}).}
\and
Maarten V. de Hoop\thanks{Department of Computational and Applied Mathematics 
  and Department of Earth Science, Rice University, 6100 Main Street, 
  Houston TX 77005, USA (\href{mailto:mdehoop@rice.edu}{\texttt{mdehoop@rice.edu}}).
  The research of this author is supported in part by the Simons Foundation.}
\and
Florian Faucher\thanks{INRIA Bordeaux Sud-Ouest Research Center, Team Project Magique-3D, France (\href{mailto:florian.faucher@inria.fr}{\texttt{florian.faucher@inria.fr}}).}
\and
Otmar Scherzer\thanks{Computational Science Center, University of
  Vienna, Oskar-Morgenstern Platz 1, A-1090 Vienna, Austria 
  (\href{mailto:otmar.scherzer@univie.ac.at}{\texttt{otmar.scherzer@univie.ac.at}}).
  The research of this author is supported by the Austrian Science Fund 
  (FWF), Project P26687-N25 Interdisciplinary Coupled Physics Imaging.}
}

\begin{document}
%% -------------------------------------------------------------------
%% -------------------------------------------------------------------

\maketitle
\slugger{sima}{xxxx}{xx}{x}{x--x}
% slugger should be set to mms, siap, sicomp, sicon, sidma, 
%                 sima, simax, sinum, siopt, sisc, or sirev

%% abstract ----------------------------------------------------------
%% -------------------------------------------------------------------
\begin{abstract}
We study the inverse boundary value problem for the Helmholtz equation
using the Dirichlet-to-Neumann map at selected frequencies as the data. 
A conditional
Lipschitz stability estimate for the inverse problem holds in the 
case of wavespeeds that are a linear combination of piecewise constant
functions (following a domain partition) and gives a framework
in which the scheme converges. 
The stability constant grows exponentially 
as the number of subdomains in the domain partition
increases. 
We establish an order optimal upper bound for the stability constant. 
We eventually realize computational experiments to demonstrate the 
stability constant evolution for three dimensional 
wavespeed reconstruction.
\end{abstract}

%% keywords ----------------------------------------------------------


\pagestyle{myheadings}
\thispagestyle{plain}
\markboth{E. BERETTA, M. DE HOOP, F. FAUCHER, AND O. SCHERZER}
         {LIPSCHITZ STABILITY OF HELMHOLTZ INVERSE PROBLEM}

%% new section -------------------------------------------------------
%% -------------------------------------------------------------------
\section{Introduction}

In this paper we study the inverse boundary value problem for the
Helmholtz equation using the Dirichlet-to-Neumann map 
at selected frequencies as the
data. This inverse problem arises, for example, in reflection
seismology and inverse obstacle scattering problems for
electromagnetic waves \cite{Bao2005a, Symes2009, Bao2010}. We consider
wavespeeds containing discontinuities.

Uniqueness of the mentioned inverse boundary value problem was
established by Sylvester \& Uhlmann \cite{Sylvester1987} assuming that
the wavespeed is a bounded measurable function. This inverse problem
has been extensively studied from an optimization point of view. We
mention, in particular, the work of \cite{Hadj-Ali2008}.

It is well known that the logarithmic character of stability of the
inverse boundary value problem for the Helmholtz equation
\cite{Alessandrini1988, Novikov2011} cannot be avoided, see also 
\cite{Hahner2001,Hohage1997}. In fact, in
\cite{Mandache2001} Mandache proved that despite of regularity 
a priori assumptions of any order on the unknown wavespeed,
logarithmic stability is the best possible. However, conditional Lipschitz
stability estimates can be obtained: accounting for discontinuities,
such an estimate holds if the unknown wavespeed is a finite linear
combination of piecewise constant functions with an underlying known
domain partitioning \cite{Beretta2012}. It was
obtained following an approach introduced by Alessandrini and Vessella
\cite{Alessandrini2005} and further developed by Beretta and Francini
\cite{Beretta2011} for Electrical Impedance Tomography (EIT) 
based on the use of singular solutions. If, on one hand, this method allows 
to use partial data, on the other hand it does not allow to find an 
optimal bound of the stability constant.  Here, we
revisit the Lipschitz stability estimate for the full
Dirichlet-to-Neumann map using complex geometrical optics (CGO) solutions
which give rise to a sharp upper bound of the Lipschitz constant in
terms of the number of subdomains in the domain partitioning. We
develop the estimate in $L^2(\Omega)$.

Unfortunately, the use of CGO's solutions leads naturally to a 
dependence of the stability constant on frequency of exponential type.
This is clearly far from being optimal as it is also pointed out in the 
paper of Nagayasu, Uhlmann and Wang \cite{Nagayasu2013}. There the authors 
prove a stability estimate, in terms of Cauchy data instead of the 
Dirichlet-to-Neumann map using CGO solutions. They derive a stability 
estimate consisting  of two parts: a Lipschitz stability estimate and a 
Logarithmic stability estimate. When the frequency increases the logarithmic 
part decreases while the Lipschitz part becomes dominant but with a stability 
constant which blows up exponentially in frequency. 

We can exploit the quantitative stability estimate, via a Fourier
transform, in the corresponding time-domain inverse boundary value
problem with bounded frequency data. Datchev and De Hoop \cite{Datchev2015} 
showed how to choose classes of non-smooth coefficient functions, one of which
is consistent with the class considered here, so that optimization
formulations of inverse wave problems satisfy the prerequisites for
application of steepest descent and Newton-type iterative
reconstruction methods. The proof is based on resolvent estimates for
the Helmholtz equation. Thus, one can allow approximate localization
of the data in selected time windows, with size inversely proportional
to the maximum allowed frequency. This is of importance to
applications in the context of reducing the complexity of field
data. We note that no information is lost by cutting out a (short)
time window, since the boundary source functions (and wave solutions),
being compactly supported in frequency, are analytic with respect to
time. We cannot allow arbitrarily high frequencies in the data. This
restriction is reflected, also, in the observation by Blazek, Stolk \&
Symes \cite{Blazek2013} that the adjoint equation, which appears in 
the mentioned iterative methods, does not admit solutions.

As a part of the analysis, we study the Fr\'{e}chet differentiability of the
direct problem and obtain the frequency and domain partitioning
dependencies of the relevant constants away from the Dirichlet
spectrum. Our results hold for finite fixed frequency data including
frequencies arbitrarily close to zero while avoiding Dirichlet
eigenfrequencies; in view of the estimates, inherently, there is a
finest scale which can be reached.
Finally we estimate the stability numerically and demonstrate
the validity of the bounds, in particular in the context
of reflection seismology.

%% new section -------------------------------------------------------
%% -------------------------------------------------------------------
\section{Inverse boundary value problem with the Dirichlet-to-Neumann
         map as the data}  \label{sec:ibvp_dtn}
         
\subsection{Direct problem and forward operator}
We describe the direct problem and some properties of the data, that
is, the Dirichlet-to-Neumann map. We will formulate the direct problem
as a nonlinear operator mapping $F_{\omega}$ from $L^\infty(\Omega)$ to
$\mathcal{L}(H^{1/2}(\partial \Omega), H^{-1/2}(\partial \Omega))$
defined as
\[
   F_{\omega}(c^{-2})=\Lambda_{\omega^2 c^{-2}} ,
\]
where $\Lambda_{\omega^2 c^{-2}}$ indicates the Dirichlet to Neumann
operator. Indeed, at fixed frequency $\omega^2$, we consider the
boundary value problem,
\begin{equation} \label{Helmholtz}
\left\{\begin{array}{rl}
   (-\Delta - \omega^2 c^{-2}(x)) u = &  0 ,\quad  \mbox{ in }\Omega ,\\
   u = & g \quad \mbox{ on } \partial \Omega ,
   \end{array}
   \right.
\end{equation}
while $\Lambda_{\omega^2 c^{-2}} :\ g \to \frac{\partial u}{\partial\nu} |_{\partial\Omega}$, where $\nu$ denotes the outward unit normal vector to $\partial\Omega$. In
this section, we will state some known results concerning the
well-posedness of problem \cref{Helmholtz} (see, for example, 
\cite{Gilbarg1983})
and regularity properties of the nonlinear map $F_{\omega}$. We will
sketch the proofs of these results because we will need to keep track
of the dependencies of the constants involved on frequency. We invoke

\medskip\medskip

\begin{assumption}\label{Aprioribound1}
There exist two positive constants $B_1, B_2$ such that
\begin{equation}
   B_1 \leq c^{-2} \leq B_2\quad\text{in }\Omega.
\end{equation}
\end{assumption}
In the sequel of Section 2 $C = C(a, b,c,\dots)$  indicates that $C$ depends only on the parameters $a, b,c,\dots$ and  we will indicate different constants with the same letter $C$.
\medskip\medskip

\begin{proposition}\label{2-energy}
Let $\Omega$ be a bounded Lipschitz domain in $\mathbb{R}^3$, $f \in
L^{2}(\Omega)$, $g \in H^{1/2}(\partial\Omega)$ and $c^{-2}\in
L^{\infty}(\Omega)$ satisfying \cref{Aprioribound1}. Then, there
exists a discrete set $\Sigma_{c^{-2}} := \{ \tilde\lambda_n\quad
|\tilde\lambda_n > 0 ,\ \forall n \in \mathbb{N}\}$ such that, for every 
$\omega^2 \in \mathbb{C}\backslash \Sigma_{c^{-2}}$, there exists a
unique solution $u \in H^1(\Omega)$ of
\begin{equation} \label{Helmholtz1}
\left\{\begin{array}{rl}
   (-\Delta - \omega^2 c^{-2}(x)) u
                = &  f\ \ \quad  \mbox{ in } \Omega ,\\
   u = & g \quad \mbox{ on } \partial \Omega .
   \end{array}
   \right.
\end{equation}
Furthermore, there exists a positive constant $C$ such that
\begin{equation}\label{energy_est1}
\| u \|_{H^{1}(\Omega)} \le C\left(1+\frac{\omega^2}{\textrm{d}(\omega^2, \Sigma_{c^{-2}})}\right) \left( \| g \|_{H^{1/2}(\partial \Omega)} + \| f \|_{L^{2}(\Omega)} \right) ,
\end{equation}
where  $C=C(\Omega,B_2)$ and $\textrm{d}(\omega^2, \Sigma_{c^{-2}})$ indicates the distance of $\omega^2$ from $\Sigma_{c^{-2}}$.
\end{proposition}

\begin{proof}
We first prove the result for $g=0$. Consider the linear operators
$-\Delta:H^1_0(\Omega)\rightarrow H^{-1}(\Omega)$ and the multiplication operator 
\begin{equation}
\label{eq:mult}
\begin{aligned}
M_{c^{-2}}: L^2(\Omega) &\rightarrow L^2(\Omega)\,,\\
u &\rightarrow c^{-2}u
\end{aligned}
\end{equation}
respectively. We can now consider the operator $K=\Delta^{-1}M_{c^{-2}}: H^1_0(\Omega)\rightarrow H^1_0(\Omega)$.  The equation
\[(-\Delta - \omega^2 c^{-2}(x)) u
                =  f\;.
 \] for $u\in H_0^1(\Omega)$ is equivalent to
 \begin{equation}\label{opeq}
 (I-\omega^2K)u=\Delta^{-1}f
  \end{equation}
 Note that $ K: H^1_0(\Omega)\rightarrow H^1_0(\Omega)$ is compact by Rellich--Kondrachov compactness theorem.  
Furthermore, by \cref{Aprioribound1} and the properties of $\Delta^{-1}$  
it follows that $K$ is self-adjoint and positive. Hence,
$K$ has a discrete set of positive eigenvalues $\{ \alpha_n
\}_{n\in\mathbb{N}}$ such that $\alpha_n \rightarrow 0$ as $n
\rightarrow \infty$. Let $\tilde\lambda_n := \frac{1}{\alpha_n}, n
\in\mathbb{N}$ and define 
$\Sigma_{c^{-2}} := \{
\tilde{\lambda}_n\ :\ n \in \mathbb{N}\}$ and let $\omega^2 \in
\mathbb{C} \backslash \Sigma_{c^{-2}}$, and show that it satisfies the assumptions of this proposition. 
Then, by the Fredholm
alternative, there exists a unique solution $u \in H^1_0(\Omega)$ of
\cref{opeq}.

To prove estimate \cref{energy_est1} we observe that
$$
u = \sum_{n=1}^{\infty} \langle u,e_n \rangle e_n ,\quad
K u =\sum_{n=1}^{\infty}\alpha_n \langle u,e_n \rangle e_n
$$
where $\{e_n\}_{n\in\mathbb{N}}$ is an orthonormal basis of $L^2(\Omega)$.
Hence we can rewrite \cref{opeq} in the form
$$
\sum_{n=1}^{\infty}(1-\omega^2\alpha_n) \langle u,e_n \rangle e_n
=\sum_{n=1}^{\infty} \langle h,e_n \rangle e_n \text{ where } h =\Delta^{-1}f
$$
Hence,
$$
   \langle u,e_n \rangle
    = \frac{1}{1-\frac{\omega^2}{\tilde\lambda_n}}
    \langle h, e_n \rangle,\quad\forall n\in\mathbb{N}
$$
and
$$
   u=\sum_{n=1}^{\infty}
     \frac{1}{1-\frac{\omega^2}{\tilde\lambda_n}}
        \langle h,e_n \rangle e_n
$$
so that
\begin{equation}\label{est_u}
\|u\|_{L^2(\Omega)}\leq \left(1+\frac{\omega^2}{d(\omega^2, \Sigma_{c^{-2}})}\right)\|h\|_{L^2(\Omega)}\leq C\left(1+\frac{\omega^2}{d(\omega^2, \Sigma_{c^{-2}})}\right)\|f\|_{L^2(\Omega)}
\end{equation}
where $C=C(\Omega, B_2)$.

Now, by multiplying equation \cref{Helmholtz1} with $u$, integrating by parts, using Schwartz' inequality, \cref{Aprioribound1,est_u} it follows in the case $g=0$:
\begin{equation}\label{est_grad}
\|\nabla u\|_{L^2(\Omega)}\leq C\left(1+\frac{\omega^2}{d(\omega^2, \Sigma_{c^{-2}})}\right)\|f\|_{L^2(\Omega)}
\end{equation}
Hence, by \cref{est_u,est_grad} we finally get
$$
\|u\|_{H^1(\Omega)}\leq C\left(1+\frac{\omega^2}{d(\omega^2, \Sigma_{c^{-2}})}\right)\|f\|_{L^2(\Omega)}.
$$
If $g$ is not identically zero then we reduce the problem to the previous case
by considering $v=u-\tilde g$ where $\tilde g\in H^1(\Omega)$ is such
that $\tilde g=g$ on $\partial\Omega$ and $\|\tilde
g\|_{H^1(\Omega)}\leq \|g\|_{H^{1/2}(\partial\Omega)}$ and we derive easily the estimate
\[
\|u\|_{H^1(\Omega)}\leq C\left(1+\frac{\omega^2}{d(\omega^2, \Sigma_{c^{-2}})}\right)(\|f\|_{L^2(\Omega)}+\|g\|_{H^{1/2}(\partial\Omega)})
\]
which concludes the proof. \qquad 
\end{proof}
 
The constants appearing in the estimate of \cref{2-energy} depends on $c^{-2}$ 
and $ \Sigma_{c^{-2}}$ which are
unknown. To our purposes it would be convenient to have constants
depending only on a priori parameters $B_1$, $B_2$ and other known
parameters. Let us denote by $\Sigma_0$ the spectrum of $-\Delta$. Then, we have the following
%\textbf{[Q1: if we make $\omega$ complex, can't we avoid
 %$\Sigma_{c^{-2}}$?]}

\medskip\medskip

\begin{proposition}\label{uniformconstants}
Suppose that the assumptions of \cref{2-energy}  are satisfied. Let $\{\lambda_n\}_{n\in\mathbb{N}}$ denote the Dirichlet eigenvalues of $-\Delta$. Then, for any $n\in\mathbb{N}$,
\begin{equation} \label{eigenvaluebound}
\frac{\lambda_n}{B_2}\leq\tilde\lambda_n\leq\frac{\lambda_n}{B_1} .
 \end{equation}
If $\omega^2$ is such that,
 \begin{equation}\label{smallfr}
 0<\omega^2<\frac{\lambda_1}{B_2},
 \end{equation}
 or, for some $n\geq 1$,
 \begin{equation}\label{higherfr}
 \frac{\lambda_n}{B_1}< \omega^2< \frac{\lambda_{n+1}}{B_2} ,
 \end{equation}
then there exists a unique solution $u\in H^1(\Omega)$ of Problem \cref{Helmholtz} and the following estimate holds
\[
\| u \|_{H^{1}(\Omega)} \le C \left( \| g \|_{H^{1/2}(\partial \Omega)} + \| f \|_{L^{2}(\Omega)} \right) ,
\]
where  $C=C(B_1,B_2,\omega^2,  \Sigma_0)$.
\end{proposition}

\begin{proof}
To derive estimate \cref{eigenvaluebound} we consider the Rayleigh
quotient related to equation \cref{Helmholtz}
$$
\frac{\int_{\Omega}|\nabla v|^2}{\int_{\Omega}c^{-2}v^2} .
$$
By \cref{Aprioribound1}, for any non trivial $v\in
H^1_0(\Omega)$ we have
$$
\frac{1}{B_2}\frac{\int_{\Omega}|\nabla v|^2}{\int_{\Omega}v^2}\leq
\frac{\int_{\Omega}|\nabla v|^2}{\int_{\Omega}c^{-2}v^2}\leq
\frac{1}{B_1}\frac{\int_{\Omega}|\nabla v|^2}{\int_{\Omega}v^2}\;.
$$

Now, we apply Courant-Rayleigh minimax principle (see for instance \cite[Theorem 4.5.1]{Dav95}, where the infinite dimensional 
Courant-Rayleigh minimax principle has been considered): The following arguments are similar as in the simple one-dimensional 
Example of Davies' book \cite[Example 4.6.1]{Dav95}.
Due to \cref{Aprioribound1} the Hilbert space
\begin{equation*}
 L_c^2 (\Omega) = \{ v : \int_{\Omega}c^{-2}v^2 < \infty\}\,, 
\end{equation*}
with norm $\norm{v}_{L_c^2} = \int_\Omega v^2 c^{-2}$ is equivalent to $L^2(\Omega)$.
\begin{equation*}
 \begin{aligned}
  \tilde{\lambda}_n &:= \inf_{\{\tu_1,\cdots,\tu_n \in H^1_0(\Omega)\}} 
  \sup_{v \in \text{span} \{\tu_1,\cdots,\tu_n\}: \norm{v}_{L^2_c} \leq 1} \frac{\int_{\Omega}|\nabla v|^2}{\int_{\Omega}c^{-2}v^2 }\,,\\
  \lambda_n &:= \inf_{\{u_1,\cdots,u_n \in H^1_0(\Omega)\}} 
  \sup_{v \in \text{span} \{u_1,\cdots,u_n\} : \norm{v}_{L^2} \leq 1} \frac{\int_{\Omega}|\nabla v|^2}{\int_{\Omega}v^2}\;.
 \end{aligned}
\end{equation*}
Note that $\norm{v}_{L^2} \leq 1$ implies that $\norm{v}_{L_c^2}^2 \leq B_2$ and that $L_c^2(\Omega)=L^2(\Omega)$.
Therefore
\begin{equation*}
  \lambda_n \leq \inf_{\{u_1,\cdots,u_n \in H^1_0(\Omega)\}} 
  \sup_{v \in \text{span} \{u_1,\cdots,u_n\} : \norm{v}_{L_c^2}^2 \leq B_2} \frac{\int_{\Omega}|\nabla v|^2}{\int_{\Omega}v^2}\;.
\end{equation*}
Now, using the scale invariance of $\frac{\int_{\Omega}|\nabla v|^2}{\int_{\Omega}v^2}$ and that $c^{-2}\leq B_2$, we get
\begin{equation*}
  \lambda_n \leq B_2 \inf_{\{u_1,\cdots,u_n \in H^1_0(\Omega)\}} 
  \sup_{v \in \text{span} \{u_1,\cdots,u_n\} : \norm{v}_{L_c^2} \leq 1} \frac{\int_{\Omega}|\nabla v|^2}{\int_{\Omega}c^{-2} v^2} = B_2 \tilde{\lambda}_n\;.
\end{equation*}
To get lower bound estimate for $\tilde{\lambda}_n$ observe that if $\norm{v}_{L_c^2} \leq 1$ then 
$\norm{v}_{L^2}^2 \leq \frac{1}{B_1}$. Hence
\begin{equation*}
  \tilde \lambda_n \leq \inf_{\{\tu_1,\cdots,\tu_n \in H^1_0(\Omega)\}} 
  \sup_{v \in \text{span} \{\tu_1,\cdots,\tu_n\} : \norm{v}_{L^2}^2 \leq \frac{1}{B_1}} 
  \frac{\int_{\Omega}|\nabla v|^2}{\int_{\Omega}c^{-2} v^2}\;.
\end{equation*}
Now, using the scale invariance of $\frac{\int_{\Omega}|\nabla v|^2}{\int_{\Omega}v^2}$ and that $c^{-2}\geq B_1$, we get
\begin{equation*}
  \tilde\lambda_n \leq \inf_{\{\tu_1,\cdots,\tu_n \in H^1_0(\Omega) \}} 
  \sup_{v \in \text{span} \{\tu_1,\cdots,\tu_n\}: \norm{v}_{L^2} \leq 1} \frac{\int_{\Omega}|\nabla v|^2}{\int_{\Omega}c^{-2} v^2} = \frac{1}{B_1} \lambda_n\;.
\end{equation*}
Thus we have shown that
$$
\frac{\lambda_n}{B_2}\leq \tilde{\lambda}_n\leq \frac{\lambda_n}{B_1},\quad \forall n\in\mathbb{N}.
$$
Hence, we have well-posedness of problem \cref{Helmholtz} if we select
an $\omega^2$ satisfying \cref{smallfr} or \cref{higherfr} and the claim follows.
\end{proof}

We observe that in order to derive the uniform estimates of 
\cref{uniformconstants} we need to assume that either 
the frequency is small \cref{smallfr} or that the oscillation of $c^{-2}$ is sufficiently small \cref{higherfr}. 
This observation can also been 
found in Davies' book \cite{Dav95}.

\medskip\medskip

In the seismic application we have in mind we might know the spectrum 
of some reference wavespeed $c_0^{-2}$. The following local result holds
\medskip\medskip

\begin{proposition}\label{continuity}
Let $\Omega$ and $c_0^{-2}$ satisfy the assumptions of
\cref{2-energy} and let
$\omega^2\in\mathbb{C}\backslash\Sigma_{c_0^{-2}}$ where
$\Sigma_{c_0^{-2}}$ is the Dirichlet spectrum of equation
\cref{Helmholtz} corresponding to $c_0^{-2}$. Then, there exists $\delta=\delta(\Omega,\omega^2, B_2,\Sigma_{c_0^{-2}})>0$ such
that, if
\[
\|c^{-2}-c_0^{-2}\|_{L^{\infty}(\Omega)}\leq \delta,
\]
then $\omega^2\in \mathbb{C}\backslash\Sigma_{c^{-2}}$ and the solution $u$ of Problem \cref{Helmholtz1} corresponding to $c^{-2}$ satisfies
\[
\| u \|_{H^{1}(\Omega)} \le C\left(1+\frac{\omega^2}{\textrm{d}(\omega^2, \Sigma_{c_0^{-2}})}\right) \left( \| f \|_{L^{2}(\Omega)} + \| g \|_{H^{1/2}(\partial \Omega)} \right),
\]

 $C=C(\Omega, B_2)$.

\end{proposition}

\begin{proof}
Let $\delta_c:=c^{-2}-c_0^{-2}$ and consider  $u_0\in H^{1}(\Omega)$ the unique solution of \cref{Helmholtz1} for $c_0^{-2}$ and consider the problem
\begin{equation}\label{auxiliarypr}
\left\{
\begin{array}{rl}
-\Delta v-\omega^2c_0^{-2}v-\omega^2\delta_c v= &\omega^2u_0\delta_c \quad \text{in } \Omega, \\
v= & 0, \quad \text{on } \partial \Omega .
\end{array}
\right.
\end{equation}
Let now 
\begin{equation*}
  L :=-\Delta-\omega^2c_0^{-2}
\end{equation*}
then, by assumption, it is invertible from $H_0^1(\Omega)$ to $L^2(\Omega)$ and we can rewrite
problem \cref{auxiliarypr} in the form
\begin{equation}\label{auxiliaryeq}
(I-K)v=h ,
\end{equation}
where $K=\omega^2L^{-1}M_{\delta_c}$ and $M_{\delta_c}$ is the
multiplication operator defined in \cref{eq:mult} and $h=L^{-1}(\omega^2u_0\delta_c)$. Observe
now that from \cref{energy_est1} $\|L^{-1}\|\leq C(1+\frac{\omega^2}{d_0})$ with $C=C(\Omega,B_2)$ and where $d_0=\textrm{dist}(\omega^2, \Sigma_{c_0^{-2}})$. Hence, we derive
\[
\|K\|\leq \omega^2\|L^{-1}\|\|M_{\delta_c}\|\leq \omega^2\|L^{-1}\| \delta
   \leq C \omega^2(1+\frac{\omega^2}{d_0})\delta.
\]
Hence, choosing $\delta=\frac{1}{2}(
C\omega^2(1+\frac{\omega^2}{d_0}))^{-1}$  the bounded operator $K$ has norm smaller than one. Hence, $I-K$ is invertible and there exists
a unique solution $v$ of \cref{auxiliaryeq} in $H^1_0$ satisfying
\cref{energy_est1} with $C=C (B_2, \omega^2, \Omega, d_0)$ and since
$u=u_0+v$ the statement follows. \qquad 
\end{proof}

\medskip\medskip

Let $\omega^2$ be such that either 
\begin{equation*}
 0 < \omega^2 < \frac{\lambda_1}{B_2},
\end{equation*}
or for some $n \geq 1$ 
\begin{equation*}
 \frac{\lambda_{n}}{B_1} < \omega^2 < \frac{\lambda_{n+1}}{B_2},
\end{equation*}
and let 
\begin{equation*}
 {\mathcal W} := \{ c^{-2} \in L^\infty(\Omega): B_1 \leq c^{-2} \leq B_2 \}\;.
\end{equation*}
Then the direct operator
\[
\begin{aligned}
  F_\omega : {\mathcal W}  & \rightarrow  \mathcal{L}(H^{1/2}(\partial \Omega), H^{-1/2}(\partial \Omega)),
  \\
  c^{-2} & \mapsto \Lambda_{\omega^2 c^{-2}},
\end{aligned}
\]
is well defined.

We will examine regularity properties of $F_\omega$ in the following lemmas. 
We will show the Fr\'echet differentiability of it.

\medskip\medskip

\begin{lemma}[Fr\'echet differentiability]\label{Frechet-diff}
Let $c^{-2} \in L^{\infty}(\Omega)$ satisfy
\cref{Aprioribound1}. Assume that $\omega^2 \in \mathbb{C}\backslash \Sigma_{c^{-2}}$. 
Then, the direct operator $F_\omega$ is Fr\'echet differentiable at
$c^{-2}$ and  its Fr\'echet derivative $DF_\omega(c^{-2})$ satisfies
\begin{equation}\label{L-bd}
\| DF_\omega[c^{-2}]\|_{\mathcal{L}(L^{\infty}(\Omega), \mathcal{L}(H^{1/2}(\partial\Omega), H^{-1/2}(\partial\Omega)))} \le C \omega^2\left(1+\frac{\omega^2}{\textrm{d}(\omega^2, \Sigma_{c^{-2}})}\right)^2
%\hat{\mathfrak{L}}_0 \omega^2 ,
\end{equation}
%where $\hat{\mathfrak{L}}_0=C_1 \left(1+\frac{\omega^2}{\textrm{d}(\omega^2, \Sigma_{c^{-2}})}\right)^2$ 
where  $C=C(\Omega, B_2)$.
\end{lemma}

\begin{proof}
Consider $c^{-2} + \delta c^{-2}$. Then, from \cref{continuity}, if $\|\delta c^{-2}\|_{L^{\infty}}(\Omega)$ is small enough, $\omega^2\notin \Sigma_{c^{-2}+\delta c^{-2}}$. An application of Alessandrini's identity then gives
\begin{equation}\label{ident-1}
\langle (\Lambda_{\omega^2(c^{-2} + {\delta c}^{-2})} - \Lambda_{\omega^2 c^{-2}})g \, , \, h \rangle = \omega^2 \int_\Omega {\delta c}^{-2} \, u v \,\dd x ,
\end{equation}
where where $\langle \cdot, \cdot \rangle$ is the dual pairing with respect
to $H^{-1/2}(\partial \Omega)$ and $H^{1/2}(\partial \Omega)$ and $u$ and $v$ solve the boundary value problems,
\[
\left\{
\begin{array}{rl}
(-\Delta - \omega^2(c^{-2} + {\delta c}^{-2}))u = & 0, \quad x\in \Omega, \\
u= & g, \quad x\in \partial \Omega,
\end{array}
\right.
\]
and
\[
\left\{
\begin{array}{rl}
(-\Delta - \omega^2 c^{-2})v = & 0, \quad x\in \Omega, \\
v = & h, \quad x\in \partial \Omega,
\end{array}
\right.
\]
respectively. We first show that the map $F_{\omega}$ is Fr\'echet differentiable and that the Fr\'echet derivative is given by 
\begin{equation}\label{Fre-DF}
\langle DF_\omega[c^{-2}]({\delta c}^{-2})  g\, , \, h\rangle = \omega^2 \int_\Omega \delta c^{-2} \, \tilde{u} v \,\dd x,
\end{equation}
where $\tilde{u}$ solves the equation
\[
\left\{
\begin{array}{rl}
(-\Delta - \omega^2 c^{-2})\tilde{u} = & 0, \quad x\in \Omega, \\
\tilde{u} = & g, \quad x\in \partial \Omega.
\end{array}
\right.
\]
In fact, by \cref{ident-1}, we have that
\begin{equation}\label{break-diff-lp}
  \langle (\Lambda_{\omega^2(c^{-2} + {\delta c}^{-2})} - \Lambda_{\omega^2 c^{-2}})g \, , \, h \rangle -  \omega^2\int_\Omega {\delta c}^{-2} \, \tilde{u} v \,\dd x
   = \omega^2 \int_\Omega {\delta c}^{-2} \, (u - \tilde{u}) v \,\dd x.
   %
%   \\
%   \le \,\, & C \omega^4 \|{\delta c}^{-2}\|_{L^\infty(\Omega)}^2 \|g\|_{H^{1/2}(\partial\Omega)} \|h\|_{H^{1/2}(\partial\Omega)}.
\end{equation}
We note that $u-\tilde{u}$ solves the equations
\[
\left\{
\begin{array}{rll}
(-\Delta - \omega^2 c^{-2})(u - \tilde{u}) = & -\omega^2 {\delta c}^{-2} \,  u , & \quad x\in \Omega, \\
u - \tilde{u} = & 0, & \quad x\in \partial \Omega.
\end{array}
\right.
\]

Using the fact that $u-\tilde u$ and $v$ are in $H^1(\Omega)$ and that $\delta c^{-2}\in L^{\infty}(\Omega)$ 
and applying Cauchy-Schwarz inequality, we get 
\begin{equation}\label{ineq:1}
 \left|\omega^2 \int_\Omega {\delta c}^{-2} \, (u - \tilde{u}) v \,\dd x \right|
  \leq \omega^2 \|{\delta c}^{-2}\|_{L^{\infty}(\Omega)}\|u - \tilde{u}\|_{L^{2}(\Omega)} \|v\|_{L^{2}(\Omega)}.
\end{equation}
Finally, using the stability estimates of \cref{2-energy} applied to $u - \tilde{u}$ and to $v$ and the stability estimates of \cref{continuity} applied to $u$ we derive
\begin{equation}\label{ineq}
 \left|\omega^2 \int_\Omega {\delta c}^{-2} \, (u - \tilde{u}) v \,\dd x \right|
  \leq C\omega^4 \left(1+\frac{\omega^2}{\textrm{d}(\omega^2, \Sigma_{c^{-2}})}\right)^3 \|{\delta c}^{-2}\|^2_{L^{\infty}(\Omega)} \|g\|_{H^{1/2}(\partial\Omega)} \|h\|_{H^{1/2}(\partial\Omega)}.
\end{equation}
Hence
\[
\begin{aligned}
~ & \left|\langle (\Lambda_{\omega^2(c^{-2} + {\delta c}^{-2})} - \Lambda_{\omega^2 c^{-2}})g \, , \, h \rangle - \omega^2 \int_\Omega {\delta c}^{-2} \, \tilde{u} v \,\dd x\right| \\
\leq & C\omega^4\left(1+\frac{\omega^2}{\textrm{d}(\omega^2, \Sigma_{c^{-2}})}\right)^3\|{\delta c}^{-2}\|^2_{L^{\infty}(\Omega)} \|g\|_{H^{1/2}(\partial\Omega)} \|h\|_{H^{1/2}(\partial\Omega)},
\end{aligned}
\]
which proves differentiability.\\
Finally by 
%\begin{equation*}
% \begin{aligned}
%    \left | \langle (\Lambda_{\omega^2 c_1^{-2}}
 %    - \Lambda_{\omega^2 c_2^{-2}}) \, u_1 |_{\partial\Omega} \, , \,
 %               u_2 |_{\partial\Omega} \rangle \right|
%   & = \left| \int_{\Omega} \omega^2 (c_1^{-2} -c_2^{-2})
%                          u_1 u_2 \, \dd x \right|
%\\
%  &\le \omega^2 \| c_1^{-2} - c_2^{-2} \|_{L^{2}(\Omega)}
        %       \|u_1\|_{L^{4}(\Omega)} \|u_2\|_{L^{4}(\Omega)}\,,
% \end{aligned}
%\end{equation*}
\[
\langle DF_\omega[c^{-2}]({\delta c}^{-2})  g\, , \, h\rangle = \omega^2 \int_\Omega \delta c^{-2} \, \tilde{u} v \,\dd x,
\]
and we get
\begin{equation*}
\begin{aligned}
\left|\langle DF_\omega[c^{-2}]({\delta c}^{-2})  g\, , \, h\rangle\right|
\leq  & \omega^2  \|{\delta c}^{-2}\|_{L^{\infty}(\Omega)}\|\tilde{u}\|_{L^{2}(\Omega)} \|v\|_{L^{2}(\Omega)} \\
\leq  & \omega^2 \left(1+\frac{\omega^2}{\textrm{d}(\omega^2, \Sigma_{c^{-2}})}\right)^2 \|{\delta c}^{-2}\|_{L^{\infty}(\Omega)} \|g\|_{H^{1/2}(\partial\Omega)} \|h\|_{H^{1/2}(\partial\Omega)}.
\end{aligned}
\end{equation*}
from which \cref{L-bd} follows. \qquad
\end{proof}

\medskip\medskip

\subsection{Conditional quantitative Lipschitz stability estimate}

Let $B_2, r_0, r_1, A, L, N$ be positive with $N \in
\mathbb{N},  N \geq 2$, $r_0 < 1$. In the sequel we will refer to these numbers as to
the a priori data. To prove the results of this section we invoke the
following common assumptions

\medskip\medskip

\begin{assumption}\label{assumption_domain}
$\Omega \subset \mathbb{R}^3$ is a bounded domain such that
\[
   |x| \leq Ar_1, \quad\forall x\in \Omega.
 \]
Moreover,
\[
   \partial\Omega \text{  of Lipschitz class with  constants } r_1
   \text{ and } L.
\]
Let $\mathcal{D}_N$ be a partition of $\Omega$ given by
\begin{equation}
   \mathcal{D}_N \triangleq \bigg\lbrace\{D_1,D_2, \dots, D_N\} \mid
   \bigcup_{j=1}^N \overline{D}_j = \Omega \,\, , \,\, (D_j \cap D_{j'})^\circ =
   \emptyset, j\neq j' \bigg\rbrace
\end{equation}
such that
\[
   \{\partial D_j\}_{j=1}^{N}
   \text{ is of Lipschitz class with  constants } r_0
   \text{ and } L.
\]
\end{assumption}

\medskip\medskip

\begin{assumption}\label{assumption_potential}
The function $c^{-2} \in \mathcal{W}_N$, that is, it satisfies
\[
B_1\leq c^{-2}\leq B_2,\quad \text{in }\Omega
\]
and is of the form
\[
   c^{-2}(x) = \sum_{j = 1}^N c_j \chi_{D_j}(x) ,
\]
where $c_j, j=1,\dots,N$ are unknown numbers and $\{D_1,\dots ,D_N\}\in \mathcal{D}_N$.
\end{assumption}

\medskip\medskip

\begin{assumption}\label{wellpos}
Assume
\[
 0<\omega^2<\frac{\lambda_1}{B_2},
\]
 or, for some $n\geq 1$,
\[
 \frac{\lambda_n}{B_1}< \omega^2< \frac{\lambda_{n+1}}{B_2}. 
\]
%0<\omega^2\notin\Sigma_{c^{-2}} ,\quad \forall c^{-2} \in \mathcal{W}_N .

\end{assumption}

Under the above assumptions we can state the following preliminary result

\medskip\medskip

\begin{lemma}\label{propHs'}
Let $\Omega$ and $ \mathcal{D}_N$ satisfy \cref{assumption_domain} and let $c^{-2}\in \mathcal{W}_N$.
Then, for every $s'\in
(0,1/2)$, there exists a positive constant $C$ with $C=C(L,s')$ such
that
\begin{equation}\label{Hsbound}
\|c^{-2}\|_{H^{s'}(\Omega)}\leq  C(L,s')\frac{1}{r_0^{s'}}\|c^{-2}\|_{L^{2}(\Omega)} .
\end{equation}
\end{lemma}

\begin{proof}
The proof is based on the extension of a result of Magnanini and Papi in \cite{Magnanini1985} to the three dimensional setting.
In fact, following the argument in \cite{Magnanini1985}, one has that
 \begin{equation}\label{bound1}
\|\chi_{D_j}\|^2_{H^{s'}(\Omega)}\leq \frac{16\pi}{(1-2s')(2s')^{1+2s'}}|D_j|^{1-2s'}|\partial D_j|^{2s'} .
\end{equation}
We now use the fact that $\{D_j\}_{j=1}^N$  is a partition of disjoint sets of $\Omega$ to show the following inequality
\begin{equation}\label{Hsbound2}
\|c^{-2}\|^2_{H^{s'}(\Omega)}\leq 2\sum_{j=1}^Nc_j^2\|\chi_{D_j}\|^2_{H^{s'}(\Omega)}
\end{equation}
In fact, in order to prove \cref{Hsbound2} recall that
\[
\|c^{-2}\|^2_{H^{s'}(\Omega)}=\int_{\Omega}\int_{\Omega}\frac{|\sum_{j=1}^Nc_j(\chi_{D_j}(x)-\chi_{D_j}(y))|^2}{|x-y|^{3+2s'}}\,dx\,dy
\]
and observe that, since the $\{D_j\}_{j=1}^N$  is a partition of disjoint sets of $\Omega$, we get
\[
|\sum_{j=1}^Nc_j(\chi_{D_j}(x)-\chi_{D_j}(y))|^2=\sum_{j=1}^Nc^2_j(\chi_{D_j}(x)-\chi_{D_j}(y))^2-\sum_{i\neq j}c_ic_j\chi_{D_i}(x)\chi_{D_j}(y)
\]
Again, by the fact that  the $\{D_j\}_{j=1}^N$  are disjoint sets, we have
\begin{eqnarray*}
& & \sum_{i\neq j}|c_ic_j|\chi_{D_i}(x)\chi_{D_j}(y)\leq \sum_{i\neq j}\frac{c_i^2+c_j^2}{2}\chi_{D_i}(x)\chi_{D_j}(y) \\
& = & \sum_{i\neq j}\frac{c_i^2}{2}(\chi_{D_i}(x)-\chi_{D_i}(y))^2\chi_{D_i}(x)\chi_{D_j}(y)+ \sum_{i\neq j}\frac{c_j^2}{2}(\chi_{D_j}(x)-\chi_{D_j}(y))^2\chi_{D_i}(x)\chi_{D_j}(y)\\
& \leq &\sum_{i\neq j}\frac{c_i^2}{2}(\chi_{D_i}(x)-\chi_{D_i}(y))^2\chi_{D_j}(y)+ \sum_{i\neq j}\frac{c_j^2}{2}(\chi_{D_j}(x)-\chi_{D_j}(y))^2\chi_{D_i}(x)\\
& \leq & \sum_{i=1}^N\frac{c_i^2}{2}(\chi_{D_i}(x)-\chi_{D_i}(y))^2\sum_{j=1}^N\chi_{D_j}(y)+\sum_{j=1}^N\frac{c_j^2}{2}(\chi_{D_j}(x)-\chi_{D_j}(y))^2\sum_{i=1}^N\chi_{D_i}(y)\\
& \leq & \sum_{i=1}^N\frac{c_i^2}{2}(\chi_{D_i}(x)-\chi_{D_i}(y))^2+\sum_{j=1}^N\frac{c_j^2}{2}(\chi_{D_j}(x)-\chi_{D_j}(y))^2 \\
& = & \sum_{i=1}^Nc_i^2(\chi_{D_i}(x)-\chi_{D_i}(y))^2
\end{eqnarray*}
where we have used the fact that $\sum_{i=1}^N\chi_{D_i}\leq 1$. 
So, we have derived that
\[
|\sum_{j=1}^Nc_j(\chi_{D_j}(x)-\chi_{D_j}(y))|^2\leq 2 \sum_{j=1}^Nc^2_j(\chi_{D_j}(x)-\chi_{D_j}(y))^2
\]
from which it follows that 
\begin{eqnarray*}
\|c^{-2}\|^2_{H^{s'}(\Omega)} & = & \int_{\Omega}\int_{\Omega}\frac{|\sum_{j=1}^Nc_j(\chi_{D_j}(x)-\chi_{D_j}(y))|^2}{|x-y|^{3+2s'}}\,dx\,dy \\
& \leq & 2\int_{\Omega}\int_{\Omega}\frac{\sum_{j=1}^Nc_j^2(\chi_{D_j}(x)-\chi_{D_j}(y))^2}{|x-y|^{3+2s'}}\,dx\,dy  \\
& \leq & 2\sum_{j=1}^Nc_j^2\int_{\Omega}\int_{\Omega}\frac{(\chi_{D_j}(x)-\chi_{D_j}(y))^2}{|x-y|^{3+2s'}}\,dx\,dy
= 2\sum_{j=1}^Nc_j^2\|\chi_{D_j}\|^2_{H^{s'}(\Omega)}
\end{eqnarray*}
which proves \cref{Hsbound2}.
so that finally from \cref{Hsbound2,bound1,assumption_domain} we get
\[
\|c^{-2}\|^2_{H^{s'}(\Omega)}\leq 2\sum_{j=1}^Nc_j^2\|\chi_{D_j}\|^2_{H^{s'}(\Omega)}\leq C(s')\sum_{j=1}^Nc_j^2|D_j|\left(\frac{|\partial D_j|}{|D_j|}\right)^{2s'}\leq \frac{C(L, s')}{r_0^{2s'}}\|c^{-2}\|^2_{L^{2}(\Omega)} .
\]
\end{proof}

We are now ready to state and prove our main stability result

\medskip\medskip

\begin{proposition}\label{prop:Stab}
Assume  \cref{assumption_domain} and let $c^{-1}_1 ,c^{-1}_2\in \mathcal{W}_N$  and let
 $\omega^2$ satisfy \cref{wellpos}. Then, there exists a positive constant  $K$, depending on $A, r_1, L$, such that,
\begin{equation}\label{Lip_stab-1}
   \| c_1^{-2} - c_2^{-2}\|_{L^{2}(\Omega)} \le
       \frac{1}{\omega^2} \, e^{K(1+\omega^2B_2) (|\Omega|/r^3_0)^{{\frac{4}{7}}}} \, \,
   \| \Lambda_{\omega^2 c_1^{-2}}
        - \Lambda_{\omega^2 c_2^{-2}}
      \|_{\mathcal{L}(H^{1/2}(\partial\Omega),
                              H^{-1/2}(\partial\Omega))} .
\end{equation}
\end{proposition}

\begin{proof}
To prove our stability estimate we follow the idea of
Alessandrini of using CGO solutions but we use slightly different ones
than those introduced in \cite{Sylvester1987} and in
\cite{Alessandrini1988} to obtain better constants in the stability
estimates as proposed by \cite{Salo}.  We also use the estimates
proposed in \cite{Salo} (see Theorem 4.4) and due to \cite{Haehner1996}
concerning the case of bounded potentials.

In fact, by Theorem 4.3 of \cite{Salo}, since $c^{-2}\in
L^{\infty}(\Omega)$, $\|c^{-2}\|_{L^{\infty}(\Omega)}\leq B_2$,
there exists a positive constant $C=C(\omega^2, B_2,A,r_1)$ such that for every 
$\zeta\in\mathbb{C}^3$ satisfying $\zeta\cdot\zeta=0$ and $|\zeta|\geq
C$ the equation
\[
-\Delta u-\omega^2 c^{-2}u=0
\]
has a solution of the form
\[
u(x)=e^{ix\cdot\zeta}(1+R(x))
\]
where $R\in H^1(\Omega)$ satisfies
\[
\|R\|_{L^2(\Omega)}\leq \frac{C}{|\zeta|},\quad 
\|\nabla R\|_{L^2(\Omega)}\leq C.
\]

Let $\xi\in \mathbb{R}^3$ and let $\tilde\omega_1$ and $\tilde\omega_2$ be unit vectors of $\mathbb{R}^3$ such that $\{\tilde\omega_1,\tilde\omega_2,\xi\}$ is an orthogonal set of vectors of $\mathbb{R}^3$ . Let $s$ be a positive parameter to be chosen later and set for $k=1,2$,
\begin{equation}
\zeta_k=\left\{\begin{array}{ccc}
                  (-1)^{k-1}\frac{s}{\sqrt 2}(\sqrt{(1-\frac{|\xi|^2}{2s^2})}\tilde\omega_1+(-1)^{k-1}\frac{1}{\sqrt{2}s}\xi+\ii\tilde\omega_2) &\mbox{for}& \frac{|\xi|}{\sqrt 2s}<1 , \\
                   (-1)^{k-1}\frac{s}{\sqrt 2}((-1)^{k-1}\frac{1}{\sqrt{2}s}\xi+\ii(\sqrt{(\frac{|\xi|^2}{2s^2}-1)}\tilde\omega_1+\tilde\omega_2))& \mbox{for}& \frac{|\xi|}{\sqrt 2s}\geq 1 .
                 \end{array}
\right.
\end{equation}
Then an straightforward computation 
gives
\[\zeta_k\cdot\zeta_k=0\]
for $k=1,2$ and
\[\zeta_1+\zeta_2=\xi.\]
Furthermore, for $k=1,2$,
 \begin{equation}
|\zeta_k|=\left\{\begin{array}{ccc}
             s &\mbox{for}& \frac{|\xi|}{\sqrt 2s}<1, \\
                  \frac{|\xi|}{\sqrt2}& \mbox{for}& \frac{|\xi|}{\sqrt 2s}\geq 1.
                 \end{array}
\right.
\end{equation}
Hence,
\begin{equation}\label{zeta}
|\zeta_k|=\max\{s,\frac{|\xi|}{\sqrt 2}\}.
\end{equation}
Then, by Theorem 4.3 of \cite{Salo}, for $|\zeta_1|,|\zeta_2|\geq C_1=\max\{C_0\omega^2B_2,1\}$, with $C_0=C_0(A,r_1)$, there exist $u_1,u_2$,  solutions to $-\Delta u_k- \omega^2 c_k^{-2} u_k=0$  for $k=1,2$ respectively, of the form
\begin{equation}\label{cgosol}
u_1(x)=e^{\ii x\cdot\zeta_1}(1+R_1(x)),\quad u_2(x)=e^{\ii x\cdot\zeta_2}(1+R_2(x))
\end{equation}
 with
\begin{equation}\label{remindercgo}
\|R_k\|_{L^2(\Omega)}\leq \frac{C_0\sqrt{|\Omega|}}{s}\omega^2B_2
\end{equation}
and
\begin{equation}\label{remindergradcgo}
\|\nabla R_k\|_{L^2(\Omega)}\leq C_0\sqrt{|\Omega|}\omega^2B_2
\end{equation}
for $k=1,2$.
It is common in the literature to use estimates which contain $\sqrt{|\Omega|}$; Different estimates in terms of $|\Omega|$ are possible and just change the leading constant $C_0$.

Consider again Alessandrini's identity
\[
\int_{\Omega} \omega^2(c_1^{-2} - c_2^{-2})u_1 u_2 \dd x = \langle(\Lambda_1-\Lambda_2)u_1|_{\partial\Omega},u_2|_{\partial\Omega}\rangle,
\]
where  $u_k\in H^1(\Omega)$ is any solution of $-\Delta u_k-\omega^2 c^{-2}_k u_k=0$ and 
$\Lambda_k = \Lambda_{\omega^2c_k^{-2}}$ for $k=1,2$. Inserting the solutions \cref{cgosol}  in Alessandrini's identity we derive
\begin{equation}
\begin{aligned}
& \left|\int_{\Omega} \omega^2(c^{-2}_1-c^{-2}_2)e^{\ii\xi\cdot x}\dd x\right| 
\\
\leq & \|\Lambda_1-\Lambda_2\|\|u_1\|_{H^{1/2}(\partial\Omega)}\|u_2\|_{H^{1/2}(\partial\Omega)}
+\left|\int_{\Omega}\omega^2(c^{-2}_1-c^{-2}_2)e^{\ii\xi\cdot x}(R_1+R_2+R_1R_2)\dd x\right|
\\
\leq & \|\Lambda_1-\Lambda_2\|\|u_1\|_{H^{1}(\Omega)}\|u_2\|_{H^{1}(\Omega)}
+E(\|R_1\|_{L^2(\Omega)}+\|R_2\|_{L^2(\Omega)}+\|R_1\|_{L^4(\Omega)}\|R_2\|_{L^4(\Omega)}).
\end{aligned}  
\end{equation}
where  $E:=\|\omega^2(c^{-2}_1-c^{-2}_2)\|_{L^{2}(\Omega)}$.
By \cref{remindercgo,remindergradcgo,zeta} and since $\Omega\subset B_{2R}(0)$  we have
\[
\|u_k\|_{H^{1}(\Omega)}\leq C\sqrt{|\Omega|}(s+|\xi|) e^{Ar_1(s+|\xi|)},\quad k=1,2.
\]
Let $s\geq C_2$  so that $s+|\xi|\leq e^{Ar_1(s+|\xi|)}$. Then, for $s\geq C_3=\max(C_1,C_2)$, using \cref{remindercgo,remindergradcgo} and the standard interpolation inequality ($\|u\|_{L^4(\Omega)}\leq \|u\|_{L^6(\Omega)}^{3/4} \|u\|_{L^2(\Omega)}^{1/4}$) we get
\begin{equation}\label{bound2}
   |\omega^2(c^{-2}_1-c^{-2}_2)\,\hat{ }\,(\xi)| \leq
   C\sqrt{|\Omega|} \left( e^{4Ar_1(s+|\xi|)}\|\Lambda_1-\Lambda_2\|+\frac{\omega^2B_2E}{s} \right)
\end{equation}
where the $\omega^2 c^{-2}_k$'s  have been extended to all $\mathbb{R}^3$ by zero and $\hat{ }$ denotes the Fourier transform. 
Hence, from \cref{bound2}, we get
\[
\int_{|\xi|\leq\rho} |\omega^2(c^{-2}_1-c^{-2}_2)\,\hat{ }\,(\xi)|^2\dd\xi \leq  C |\Omega|\rho^3\left(e^{8Ar_1(s+\rho)}\|\Lambda_1-\Lambda_2\|^2+\frac{\omega^4B^2_2E^2}{s^2}\right)
\]
and hence 
\begin{equation}\label{bound3}
\begin{aligned}
\|\omega^2 (c^{-2}_1-c^{-2}_2)\,\hat{ } \,\|^2_{L^2(\mathbb{R}^3)} 
\leq  C |\Omega|\rho^3 \bigg(e^{8Ar_1(s+\rho)}\|\Lambda_1-\Lambda_2\|^2 
     &+ \frac{\omega^4B^2_2E^2}{s^2}\bigg) \\
     &+ \int_{|\xi|\geq \rho}|\omega^2 (c^{-2}_1-c^{-2}_2)\,\hat{ }\,(\xi)|^2\,\dd\xi
\end{aligned}
\end{equation}
where $C=C(A,r_1)$.
By \cref{Hsbound,Hsbound2} we have that
\[
\|\omega^2(c^{-2}_1-c^{-2}_2)\|^2_{H^{s'}(\Omega)} \leq \frac{C}{r_0^{2s'}}E^2,
\]
where $C$ depends on $L, s'$ and hence
\begin{eqnarray*}
\rho^{2s'}\int_{|\xi|\geq \rho}|\omega^2 (c^{-2}_1-c^{-2}_2)\,\hat{ }\,(\xi)|^2\,\dd\xi&\leq &\int_{|\xi|\geq \rho}|\xi|^{2s'}|\omega^2 (c^{-2}_1-c^{-2}_2)\,\hat{ }\,(\xi)|^2\,\dd\xi\\
&\leq &\int_{\mathbb{R}^3}(1+|\xi|^{2})^{s'}|\omega^2 (c^{-2}_1-c^{-2}_2)\,\hat{ }\,(\xi)|^2\,\dd\xi\leq \frac{C}{r_0^{2s'}}E^2.
\end{eqnarray*}
Hence,  we get
\[
\int_{|\xi|\geq \rho}|\omega^2 (c^{-2}_1-c^{-2}_2)\,\hat{ }\,(\xi)|^2\,\dd\xi\leq \frac{CE^2}{r^{2s'}_0\rho^{2s'}}
\]
for every $s'\in (0,1/2)$. Inserting last bound in \cref{bound3}  we derive
\[
\|\omega^2 (c^{-2}_1-c^{-2}_2)\,\hat{ } \,\|^2_{L^2(\mathbb{R}^3)}\leq C\left(\rho^3|\Omega|e^{8Ar_1(s+\rho)}\|\Lambda_1-\Lambda_2\|^2+\rho^3|\Omega| \frac{\omega^4B^2_2E^2}{s^2}+\frac{E^2}{r_0^{2s'}\rho^{2s'}}\right).
\]
where $C=C(L,s')$. To make the last two terms in the right-hand side of the inequality of equal size we pick up
\[\sqrt[3]{|\Omega|}\rho=\left(\frac{|\Omega|}{r_0^3}\right)^{\frac{2s'}{3(3+2s')}}\left(\frac{1}{\alpha}\right)^{\frac{1}{3+2s'}}s^{\frac{2}{3+2s'}}
\]
with $\alpha=\max \{1,\omega^4B_2^2\}$.  Then, by \cref{assumption_domain} and observing that we might assume without loss of 
generality that $\frac{|\Omega|}{r_0^3}>1$. In fact, if this is not the case we can choose a smaller value of $r_0$ so that the condition is satisfied.
\begin{equation*}
\|\omega^2 (c^{-2}_1-c^{-2}_2)\|^2_{L^{2}(\Omega)}\leq
CE^2\left(\frac{|\Omega|}{r_0^3}\right)^{\frac{2s'}{3+2s'}}\left(e^{C_4(\frac{|\Omega|}{r^3_0})^{\frac{2s'}{3(3+2s')}}s}\left(\frac{\|\Lambda_1-\Lambda_2\|}{E}\right)^2+\left(\frac{\alpha}{s^2}\right)^{\frac{2s'}{3+2s'}}\right)
\end{equation*}
for $s\geq C_3$ and where $C $ depends on $s', L, A,r_1$ and $C_4$ depends on $L,A,r_1$. We now make the substitution
\[
s=\frac{1}{C_4(\frac{|\Omega|}{r^3_0})^{\frac{2s'}{3(3+2s')}}}\left|\log \frac{\|\Lambda_1-\Lambda_2\|}{E}\right|
\]
where we assume that
\[
\frac{\|\Lambda_1-\Lambda_2\|}{E} < c:=e^{-\bar C\max\{1,\omega^2B_2\}(\frac{|\Omega|}{r^3_0})^{\frac{2s'}{3(3+2s')}}}
\]
with $\bar C=\bar C(R)$ in order that  the constraint $s\geq C_3$ is satisfied.  Under this assumption,
\begin{equation}\label{logstab}
\|\omega^2 (c^{-2}_1-c^{-2}_2)\|_{L^{2}(\Omega)}\leq C(\sqrt\alpha)^{\frac{2s'}{3+2s'}}\left(\frac{|\Omega|}{r^3_0}\right)^{\frac{2s'}{3+2s'}\frac{9+10s'}{6(3+2s')}}E\left(\left|\log \frac{\|\Lambda_1-\Lambda_2\|}{E}\right|^{-\frac{2s'}{3+2s'}}\right)
\end{equation}
 where $C=C( L,s',A,r_1)$ and we can rewrite last inequality in the form
\begin{equation}\label{logstab1}
E\leq C(1+\omega^2B_2)^{\frac{2s'}{3+2s'}}\left(\frac{|\Omega|}{r^3_0}\right)^{\frac{2s'}{3+2s'}\frac{9+10s'}{6(3+2s')}}E\left(\left|\log \frac{\|\Lambda_1-\Lambda_2\|}{E}\right|^{-\frac{2s'}{3+2s'}}\right)
\end{equation}
which gives
\begin{equation}\label{lip1}
E\leq e^{C(1+\omega^2B_2)(\frac{|\Omega|}{r^3_0})^{\frac{9+10s'}{6(3+2s')}}}\|\Lambda_1-\Lambda_2\|
\end{equation}
where $C=C( L,s',A,r_1)$.
On the other hand if
\[
\frac{\|\Lambda_1-\Lambda_2\|}{E}\geq c,
\]
then
\begin{equation}\label{lipstab}
\|\omega^2 (c^{-2}_1-c^{-2}_2)\|_{L^{2}(\Omega)}\leq c^{-1}\|\Lambda_1-\Lambda_2\|\leq e^{\bar C(1+\omega^2B_2)\left(\frac{|\Omega|}{r^3_0}\right)^{\frac{1}{3(3+2s')}}}\|\Lambda_1-\Lambda_2\|
\end{equation}
Hence, from \cref{lip1,lipstab} and recalling that $s'\in (0,\frac{1}{2})$, we have that
\begin{equation}\label{lip2}
E\leq e^{C(1+\omega^2B_2)(\frac{|\Omega|}{r^3_0})^{\frac{9+10s'}{6(3+2s')}}}\|\Lambda_1-\Lambda_2\|
\end{equation}
Choosing $s'=\frac{1}{4}$, we derive 
\[
\|c^{-2}_1-c^{-2}_2\|_{L^{2}(\Omega)}\leq \frac{1}{\omega^2}e^{K(1+\omega^2B_2)(|\Omega|/r^3_0)^{\frac{4}{7}}}\|\Lambda_1-\Lambda_2\|
\]
where $K=K(L,A,r_1,s')$ and the claim follows. \qquad
\end{proof}

\medskip\medskip

\begin{remark}
\label{re:stability_n}
Here we state an $L^\infty$-stability estimate, in contrast to the $L^2$-stability estimate in \cref{prop:Stab}.

Observing that
\[
\frac{1}{\sqrt{|\Omega|}}\|c^{-2}_1-c^{-2}_2\|_{L^{2}(\Omega)}\leq \|c^{-2}_1-c^{-2}_2\|_{L^{\infty}(\Omega)}\leq \frac{C}{r_0^{3/2}}\|c^{-2}_1-c^{-2}_2\|_{L^{2}(\Omega)},
\]
where $C=C(L)$, and we immediately get the following stability estimate in the
$L^{\infty}$ norm
\[
\|c^{-2}_1-c^{-2}_2\|_{L^{\infty}(\Omega)}\leq \frac{C}{\omega^2}e^{K(1+\omega^2B_2)(|\Omega|/r^3_0)^{\frac{4}{7}}}\|\Lambda_1-\Lambda_2\|
\]
with $C=C(L)$.
\end{remark}

\begin{remark} \label{rem:cubpart}
In \cite{Beretta2012} the following lower bound of the stability constant has been obtained in the case of a uniform polyhedral partition $\mathcal{D}_N$
\begin{equation}\label{eq:SClb}
C_N\geq \frac{1}{4\omega^2}e^{K_1N^{\frac{1}{5}}}
\end{equation}
Choose a uniform cubical partition $\mathcal{D}_N$ of $\Omega$ of mesh size $r_0$. Then,
\begin{equation} \label{eq:omega_link_N}
|\Omega|=Nr_0^3
\end{equation}
and estimate \cref{Lip_stab-1} of \cref{prop:Stab} gives
\begin{equation}\label{eq:SCub}
C_N= \frac{1}{\omega^2}e^{K(1+\omega^2B_2)N^{\frac{4}{7}}},
\end{equation}
which proves  a sharp bound on the Lipschitz constant with respect to $N$ when the global DtN map is known.
In \cite{Beretta2012} a Lipschitz stability estimate has been derived
in terms of the local DtN map using singular solutions.
This type of solutions allows to recover the unknown piecewise
constant wavespeeds by determining it on the outer boundary of the
domain and then, by propagating the singularity inside the domain, to
recover step by step the wavespeed on the interface of all subdomains
of the partition. This iterative procedure does not lead to sharp
bounds of the Lipschitz constant appearing in the stability
estimate. It would be interesting if one can get a better bound of the
Lipschitz constant using oscillating solutions.
\end{remark}

\medskip\medskip

\begin{remark} \label{remark:Frechetlowerbound}
In \cref{Frechet-diff} we have seen that $F_{\omega}$ is Fr\'{e}chet differentiable with Lipschitz derivative $DF_{\omega}$ for which we have derived an upper bound in terms of the apriori data. From the stability estimates we can easily derive the following lower bound 
\begin{equation}\label{Frechetlowerbound}
   \min_{c^{-2} \in \mathcal{W}_N ;\ h \in \R^N ,\ \| h \|_{L^{\infty}(\Omega)} = 1}
           \| DF_{\omega}[c^{-2}]h \|_*\geq \omega^2e^{-K(1+\omega^2B_2)(\frac{|\Omega|}{r_0^3})^{4/7}}.
\end{equation}
where $K=K(L,A,r_1)$ and $\|\cdot\|_*$ indicates the norm in $\mathcal{L}(H^{1/2}(\partial\Omega), H^{-1/2}(\partial\Omega))$ i.e.  \[
\|T\|_*=\sup\{
\langle Tg,f \rangle |
: g,f\in H^{1/2}(\partial\Omega), \|g\|_{H^{1/2}(\partial\Omega)}=\|f\|_{H^{1/2}(\partial\Omega)}=1\}\]  
In fact, by the injectivity of $DF_{\omega}$
\[
  \min_{c^{-2} \in \mathcal{W}_N ;\ h \in \R^N ,\ \| h \|_{L^{\infty}(\Omega)} = 1}
           \| DF_{\omega}[c^{-2}]h\|_*=m_0/2>0
           \]
Then, there exists $h_0$ satisfying $ \| h_0 \|_{L^{\infty}(\Omega)} = 1$ and $c_0^{-2}\in \mathcal{W}_N$ such that 
\[
 \| DF_{\omega}[c_0^{-2}] h_0 \|_*\leq m_0\;.
\]
Hence,  by the definition of $\|\cdot\|_*$  it follows that 
\[
\left| \langle DF_{\omega}[c_0^{-2}](h_0)g,f \rangle \right|=\left|\int_{\Omega}h_0\tilde{u}_0v_0\right|\leq m_0\|\tilde{u}_0\|_{H^{1/2}(\partial\Omega)}\|v_0\|_{H^{1/2}(\partial\Omega)}
\]
where $\tilde{u}_0$ and $v_0$ are solutions to the equation $(-\Delta-\omega^2c_0^{-2})u=0$ in $\Omega$ with boundary data $g$ and $f$, respectively.  Proceeding like in the proof of the stability result \cref{prop:Stab,re:stability_n} we derive that
\[
1=\|h_0^{-2}\|_{L^{\infty}(\Omega)}\leq \frac{1}{\omega^2} e^{K(1+\omega^2B_2)(\frac{|\Omega|}{r_0^3})^{4/7}}m_0
\]
which gives the lower bound \cref{Frechetlowerbound}.
\end{remark}

%% new section -------------------------------------------------------
%% -------------------------------------------------------------------
\section{Computational experiments}

In this section, we numerically compute the stability constant 
for the inverse problem associated with the 
Dirichlet-to-Neumann map. We illustrate the stability behaviour 
and compare it with the analytical bounds derived in 
\cref{sec:ibvp_dtn}.
The estimates we provide here are obtained from the definition 
of the stability constant, 
\begin{equation} \label{eq:stability_direct_method}
 \Vert c_1^{-2} - c_2^{-2} \Vert^2 < 
 \mathcal{C} \Vert F_{\omega}(c_1^{-2}) - F_{\omega}(c_2^{-2}) \Vert^2,
\end{equation}
where $\Vert c_1^{-2} - c_2^{-2} \Vert$ denotes the $L^2$-norm 
of the functions from the finite dimensional Ansatz space. 
In particular we consider here a geophysical example of reconstruction
where normal data are collected on the boundary. In this situation 
$c_1$ and $c_2$ are assimilated to two different wavespeeds.
Hence the boundary value problem~\cref{Helmholtz} 
corresponds to the propagation of acoustic wave in the media 
for a boundary source $g$ using the wavespeeds $c_1$ and $c_2$
respectively. In our experiments, Gaussian
shaped (spatial) source functions (see \cref{fig:aq:src})
are applied. Then the normal data (measurements of the normal derivative
of the field) are acquired on the boundary 
in order to generate the forward operator. The numerical
stability estimates are finally obtained by the knowledge of all
quantities of equation~\cref{eq:stability_direct_method}.

In the \cref{rem:cubpart}, we have formulated the stability
constant depending on the number of cubical partitions $N$ in the 
model representation equation~\cref{eq:omega_link_N}. This 
situation is well adapted for numerical applications where the domain 
is commonly discretized. Hence we want to verify the (exponential) 
dependence of the stability constant with $N$.

\begin{figure}[h!] \centering
   \includegraphics[]{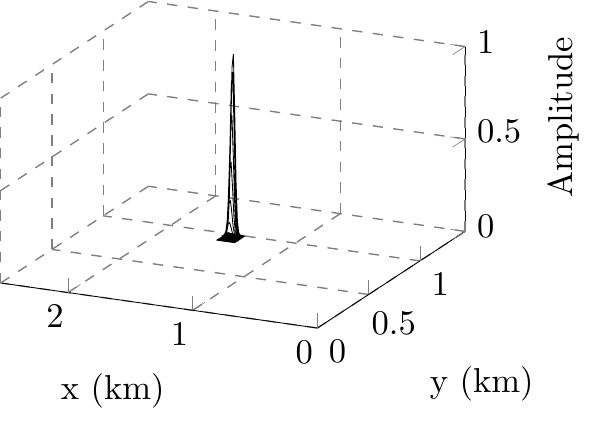}
\caption{Illustration of the source shape for a localized boundary source.}
\label{fig:aq:src}
\end{figure}

The model (assimilated to a wavespeed here)
is defined on a cubical (structured) domain 
partition of a rectangular block. 
With increasing $N$, the size of the cubes decreases, 
possibly non uniformly. We use 
piecewise constant functions on the cubes to define 
the wavespeeds following the main assumption for the 
Lipschitz stability to hold. Such a partition can be
related to Haar wavelets, where $N$ determines the scale. 
These naturally introduce
approximate representations, that is, when the
scale of the approximation is coarser than the finest scale contained
in the model.

In order to solve the forward problem, the numerical discretization 
of the operator is realized 
using discontinuous Galerkin method, 
where Dirichlet boundary conditions are invoked.
The Dirichlet sources at the top boundary introduce
Identity block in the discretized Helmholtz operator and give the
following linear problem
\begin{equation}
\left(\begin{matrix} 
        A_{i i} & A_{i \partial} \\
        A_{\partial i} & A_{\partial \partial} 
      \end{matrix}\right)
      \left(\begin{matrix} u_i \\ u_{\partial} \end{matrix}\right) =
\left(\begin{matrix} 
        A_{i i} & A_{i \partial} \\
             0  & Id
      \end{matrix}\right)
      \left(\begin{matrix} u_i \\ u_{\partial} \end{matrix}\right) =
      \left(\begin{matrix} 0   \\ g            \end{matrix}\right),
\end{equation} 
where $A$ represents the discretized operator, $i$ labels interior 
points and $\partial$ labels boundary points,
$g$ has values at the source location and is zero elsewhere.
This system verifies $u_\partial = u_{\vert \partial \Omega} = g$ 
(i.e. Dirichlet boundary condition) and 
$A_{i i} u_i + A_{i \partial} u_{\vert \partial \Omega} = 0$.
The normal derivative data are generated
by taking the normal derivative of the solution wavefield $u$ 
on the surface. \\

\begin{figure}[h!] \centering
    \subfigure[Partition using $N=2,880$ domains.]
               {\includegraphics[]{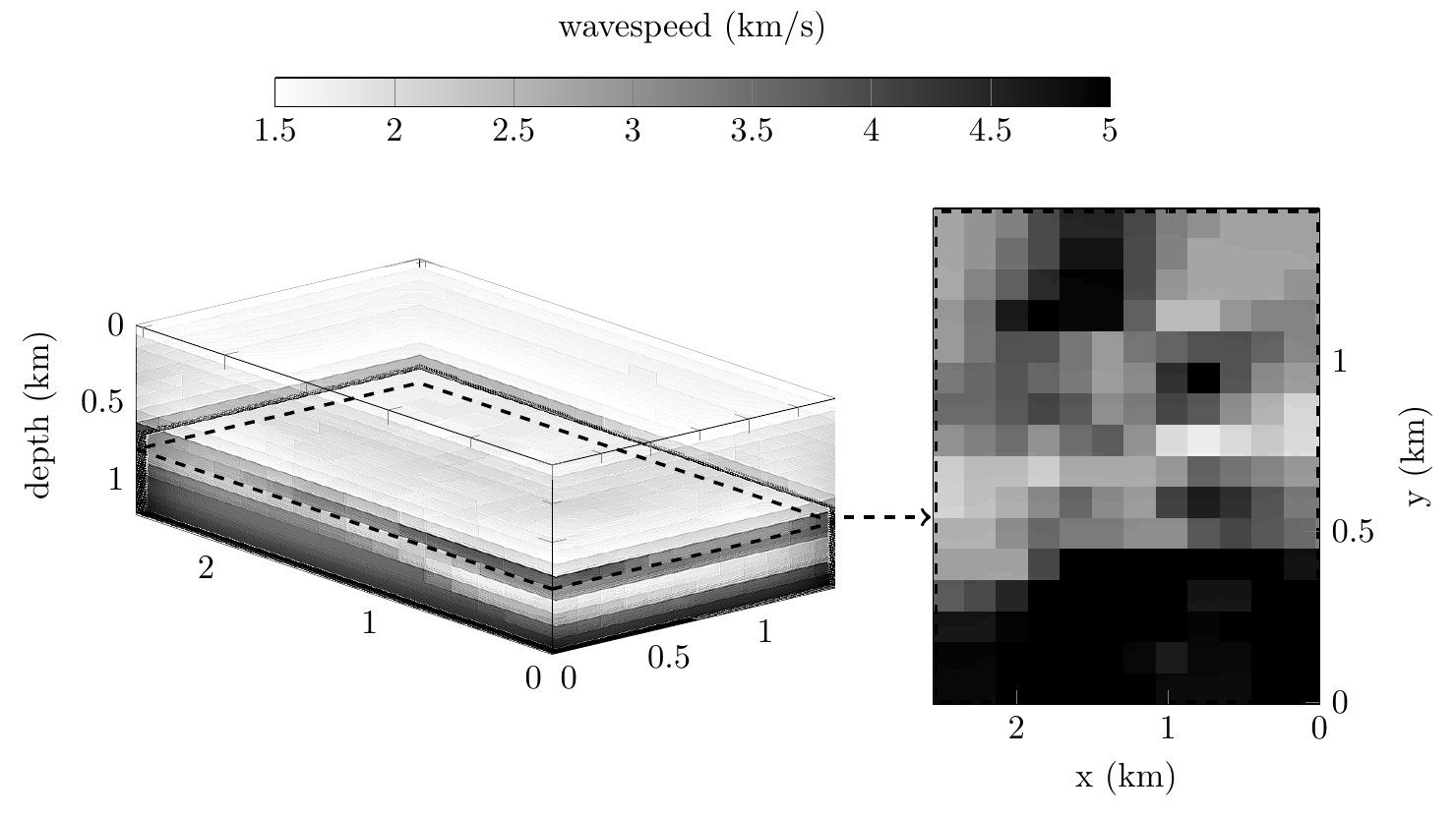}}\\
    \subfigure[Partition using $N=1,527,168$ domains.]
               {\includegraphics[]{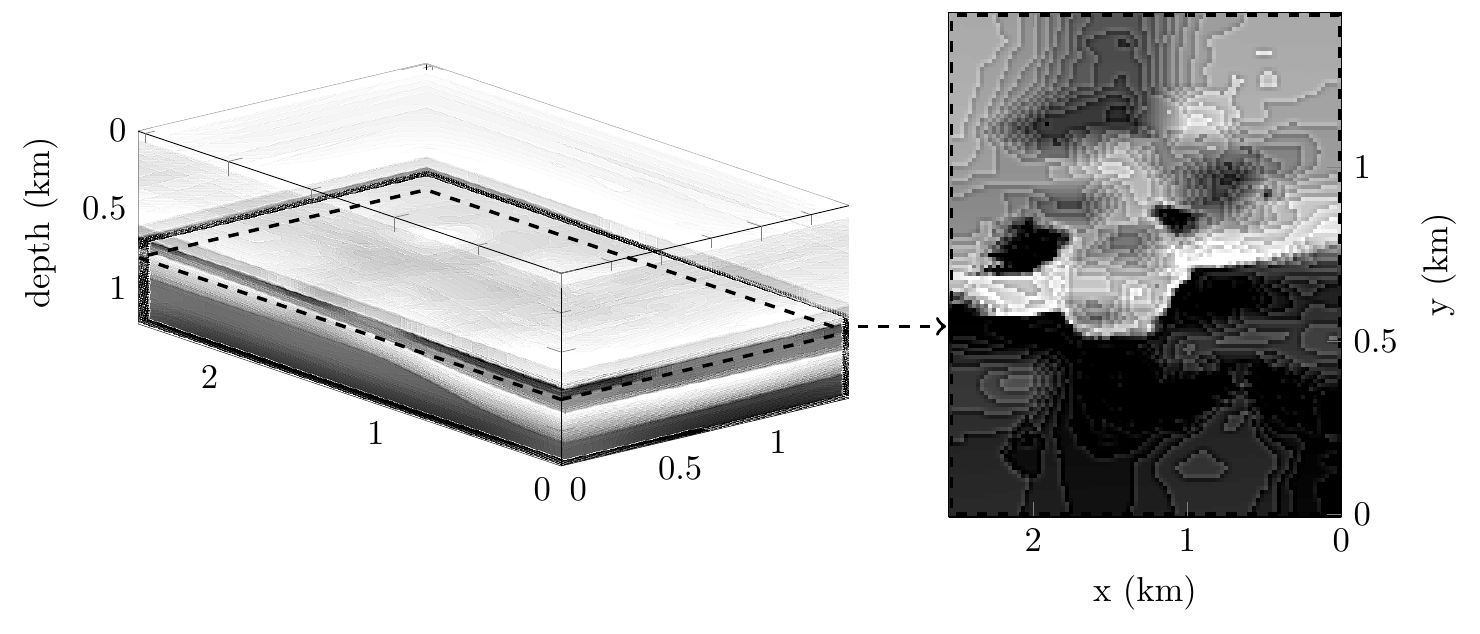}}
    \caption{Three dimensional representations 
             and horizontal sections at $800$m depth
             of the reference wavespeed ($c_1$) using different 
             partition, i.e. scales. Every scale has a 
             structured (rectangular) decomposition using 
             piecewise constant. The size of the rectangular 
             boxes defines the scale of the wavespeed.}
\label{fig:model_compression}
\end{figure}

Our experiments use a three dimensional model of size 
$2.55\times 1.45\times1.22$km. The wavesepeed $c_1$ is
viewed as a reference model (which is known in this test case)
and is represented \cref{fig:model_compression} (courtesy Statoil).
We also illustrate the different
partitions of a model and the notion of approximation. Obviously the larger
the number of subdomains is, the more precise will the representation be.

For the computation of the stability estimates we consider $c_2$ 
as the model shown in \cref{fig:model_for_stability}.
This setup can be associated
with the `true' subsurface \cref{fig:model_compression} and 
starting model \cref{fig:model_for_stability}. In this context
we have chosen the initial guess with no knowledge of any structures 
by simply considering a one dimensional variation in depth.

\begin{figure}[h!] \centering
  \includegraphics[]{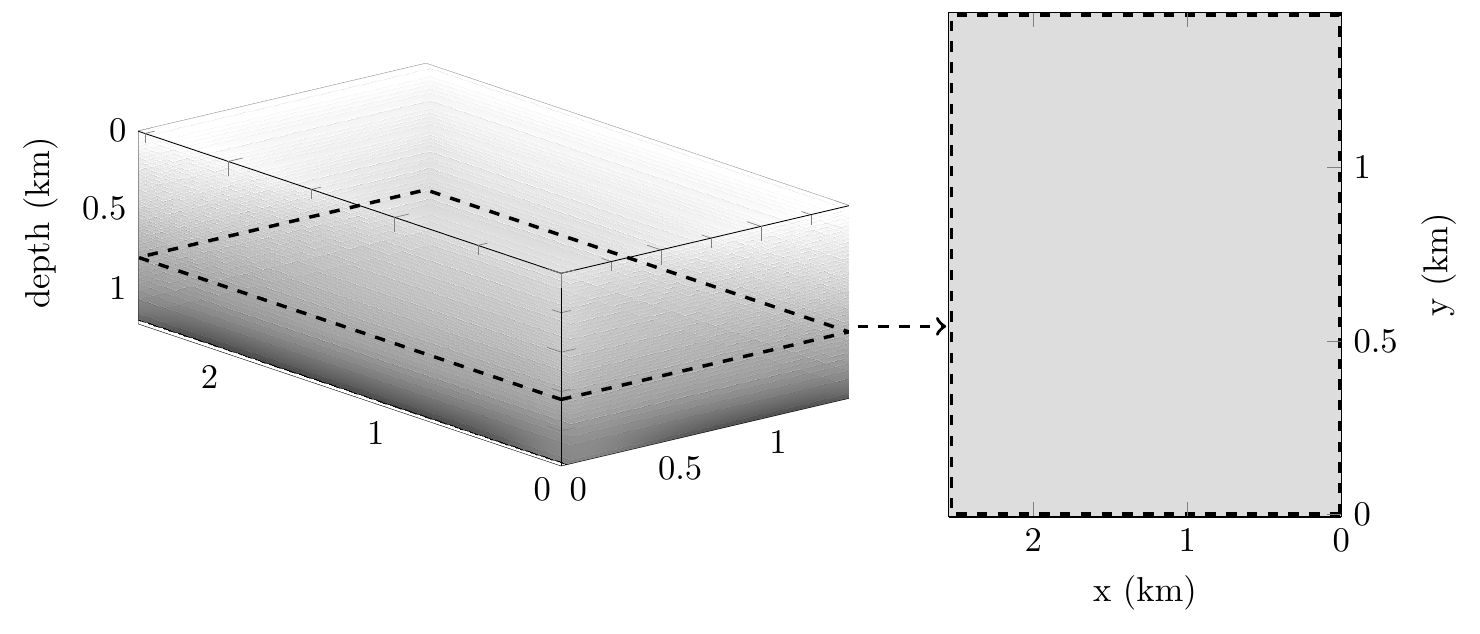}
  \caption{Three dimensional wavespeed used for the successive 
           estimation of the stability constant ($c_2$),
           $3$D representation (left) and horizontal sections 
           at $800$m depth (right).}
  \label{fig:model_for_stability}
\end{figure}

%% -----------------------------------------------------------------
%% Full data
%% -----------------------------------------------------------------
\subsection{Estimates using the full Dirichlet-to-Neumann map}
We consider the full data case where
the Gaussian sources (see \cref{fig:aq:src}) are positioned on 
each surface following a regular map. For each source, the 
data are acquired all over the boundary. We introduce a 
total of $630$ sources and $76 538$ data points for each.

At a selected partition (number of domains) and frequency, 
we simulate the data for the two media $c_1$ and $c_2$ and 
compute the difference, from which we deduce the 
stability constant following 
equation~\cref{eq:stability_direct_method}.The main 
difference with the standard seismic setup is that we consider 
data on all the boundary and not only at the top. This last case will 
be mentioned in the \cref{subsec:partialdata}.\\

The numerical estimates for the stability constant 
$\mathcal{C}$ should depend on the number of domains $N$ 
following the expression of the lower and upper 
bounds defined in the \cref{rem:cubpart},
equations~\cref{eq:SClb} and \cref{eq:SCub}. Thus
we fix the frequency and estimate the stability 
for different partitions. The evolution of the estimates 
and underlying bounds are presented in
\cref{fig:stability_depend_n_direct} at two selected frequencies,
$5$ and $10$Hz. We plot on a $\log \log$ scale the function
$\log(\mathcal{C} \omega^2)$ to focus on the power of $N$
in the estimates, which is the slope of the lines ($4/7$ for the
upper bound and $1/5$ for the lower bound).\\

\begin{figure}[h!] \centering
 \subfigure[Stability estimates at $5$Hz frequency]
           {\includegraphics[]{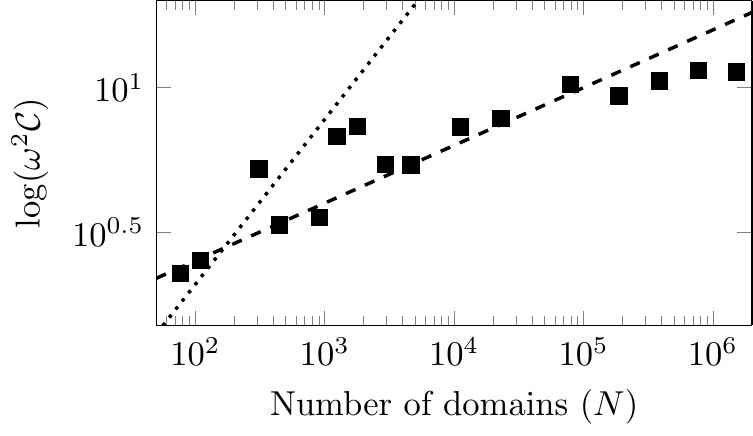}
            \label{fig:stability_depend_n_direct_5hz}} \hfill
 \subfigure[Stability estimates at $10$Hz frequency]
           {\includegraphics[]{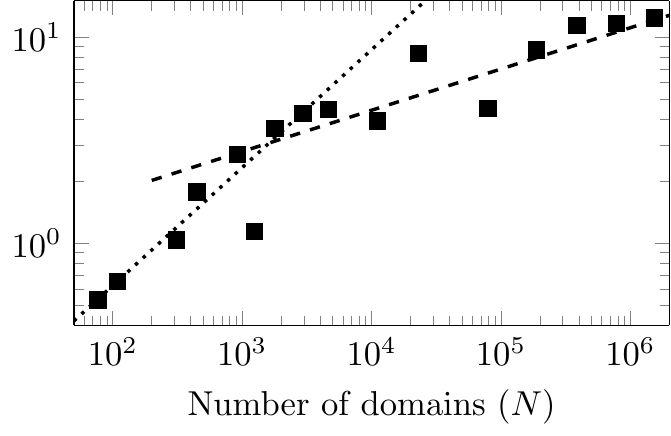}
            \label{fig:stability_depend_n_direct_10hz}}
 \caption{The black squares represent the computational 
          estimates of the stability constant 
          ($ \blacksquare $)         
          depending on the number of domains $N$ at selected frequency. 
          The dashed line ($\boldsymbol- \boldsymbol-$) 
          represents the analytical lower bound and the dotted
          line ($\boldsymbol\cdot \boldsymbol\cdot$) the
          upper bound, estimated with 
          equation~\cref{eq:numeric_cste}.}
 \label{fig:stability_depend_n_direct}
\end{figure}

Regarding the different coefficients in the analytical bounds, 
$K$ and $K_1$ remain undecided and are numerically approximated so that 
the bounds match the estimates at best. For instance the 
numerical value for $K_1$ is obtained from equation~(\cref{eq:SClb})
by computing the average value based on the numerical stability estimates
and $K$ is approximated following the same principle: 
\begin{equation} \label{eq:numeric_cste}
 K_1 = \dfrac{1}{n_{st}} 
\sum_{i=1}^{n_{st}}\dfrac{\log(4\omega^2\mathcal{C}_i)}{N_i^{1/5}},
\quad \quad \quad \quad
 K = \dfrac{1}{n_{st}} \sum_{i=1}^{n_{st}}
     \dfrac{\log(\omega^2\mathcal{C}_i)}{(1 + \omega^2B_2)N_i^{4/7}}.
\end{equation}
Here, $n_{st}$ is the number of numerical stability constant estimates 
and $\mathcal{C}_i$ the corresponding estimate for partitioning $N_i$.
We actually limit the computation of $K$ to use the first scales as it
grows too rapidly.
The numerical values obtained are given \cref{table:N_dependency_cst}. 
We also note that the term $\omega^2 B_2$ of the upper bound 
equation~\cref{eq:SCub} is relatively small in the geophysical 
context as we have here $B_2=5.10^{-7}$. \\

We can see that the stability constant increases with the number 
of subdomains, as expected. There are clearly two states in
the evolution of the estimates at the highest frequency 
(10Hz, \cref{fig:stability_depend_n_direct_10hz}). For a low number 
of partitions $N$ the numerical estimates match particularly well the upper bound 
while at finer scale it follows accurately the lower bound. This is illustrated
in \cref{fig:stability_depend_n_direct_10hz_groupplot} where we decompose
the two parts of the estimates between the low and high number of domains.

\begin{figure}[h!] \centering
   \includegraphics[]{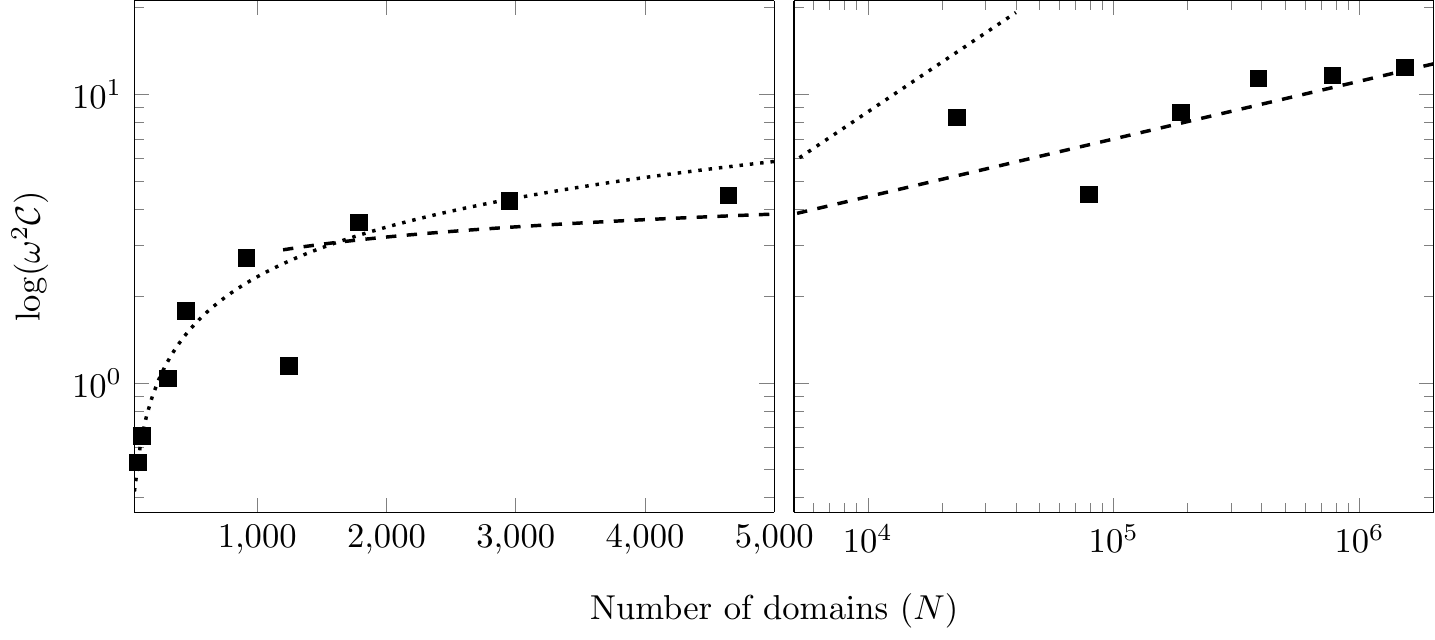}
   \caption{The black squares represent the computational 
            estimates of the stability constant 
            (\ref{plot:mark_estimates_N}) 
            depending on the number of domains $N$ at $10$Hz.
            The left part shows the coarsest scales which match 
            accurately the upper bound (dotted
            line, \ref{plot:upper_bound}).
            On the right the finer scale estimates are 
            accurately anticipated by the lower bound
            (dashed line, \ref{plot:lower_bound}).
            The constants $K$ and $K_1$ for the computation
            of the lower and upper bounds are numerically
            approximated with values given \cref{table:N_dependency_cst},
            following equation~\cref{eq:numeric_cste}.}
  \label{fig:stability_depend_n_direct_10hz_groupplot}
\end{figure}

\renewcommand{\arraystretch}{1.2}
\begin{table}[h!]
\caption{Numerical estimation of the constant in the 
         analytical bounds formulation for the numerical 
         estimates of the stability
         (\cref{fig:stability_depend_n_direct}),
         with $B_2=(1/1400)^2$).}
\label{table:N_dependency_cst}
\begin{center} 
\begin{tabular}{>{\centering\arraybackslash}p{.25\linewidth}|
                >{\centering\arraybackslash}p{.25\linewidth}|
                >{\centering\arraybackslash}p{.25\linewidth}|}
                & 5Hz & 10Hz  \\ \hline
$K_1$   & $1$    & $0.7$      \\ \hline
$K$     & $0.15$ & $0.05$     \\ \hline
\end{tabular} \end{center}
\end{table}

Alternatively for a lower frequency, i.e. $5$Hz on 
\cref{fig:stability_depend_n_direct_5hz}, the upper 
bound appears to increase too rapidly while the lower 
bound matches accurately the evolution of the 
stability constant estimates. Hence the upper bound 
we have obtained here is particularly appropriate for 
coarse scale and high frequency: when the variation of model
is much coarser compared to the wavelength.
\bigskip

%% ------------------------------------------
%% partial data
\subsection{Seismic inverse problem using partial data} \label{subsec:partialdata}

In realistic geophysical experiments for the reconstruction 
of subsurface area (seismic tomography), it is more appropriate 
not to consider the full data but partial data only located on 
the upper surface. The data obtain from $c_1$ can be seen as field 
observation (sensor measurement of a seismic event at the surface).
The data using $c_2$ are simulation using an `initial guess'.
For the reconstruction, we mention the full waveform inversion 
method, where the recovery follows an iterative minimization of 
the difference between the measurements and simulations, to successively
update the initial guess (see \cite{Tarantola1984,Pratt1998}). 
There is also the difference in the boundary conditions where 
perfectly matched layers (PMLs) or absorbing boundary conditions 
are invoked instead of the Dirichlet boundary condition for the lateral
and bottom boundaries. However the 
top boundary is a free surface and remains a Dirichlet boundary 
condition. \\

For this test case we reproduce the same experiments but 
limiting the set of sources and the collected data to be
at the top boundary only.
We define a set of sources at the surface, 
separated by $160$m along the $x-$axis and 
$150$m along $y-$axis to generate a regular map of
$16 \times 10$ points. The receivers (data location) 
are positioned in the same fashion every $60$m along 
the $x-$axis and $45$m along $y-$axis and 
generate a regular map of $43 \times 32$ points, 
see \cref{fig:aq:lattice_src,fig:aq:lattice_rcv}. The partial boundary 
data computed are illustrated for a single centered 
boundary shot at $5$Hz frequency \cref{fig:aq:data}.\\

\begin{figure}[h!] \centering
   \subfigure[the crosses represent the boundary sources locations.]
              {\label{fig:aq:lattice_src}\includegraphics[]{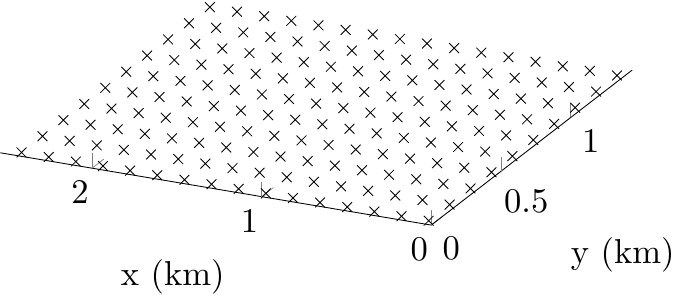}} \hfill
   \subfigure[The lattice represents 
              the discretization of the data,
              i.e. the receivers location.]
              {\label{fig:aq:lattice_rcv}\includegraphics[]{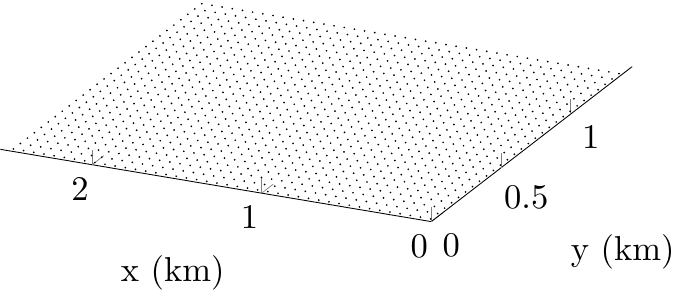}} \\
             
   \subfigure[Data recovered from a boundary centered shot, i.e.
              wavefield measured at the receivers location.]
             {\label{fig:aq:data}\includegraphics[]{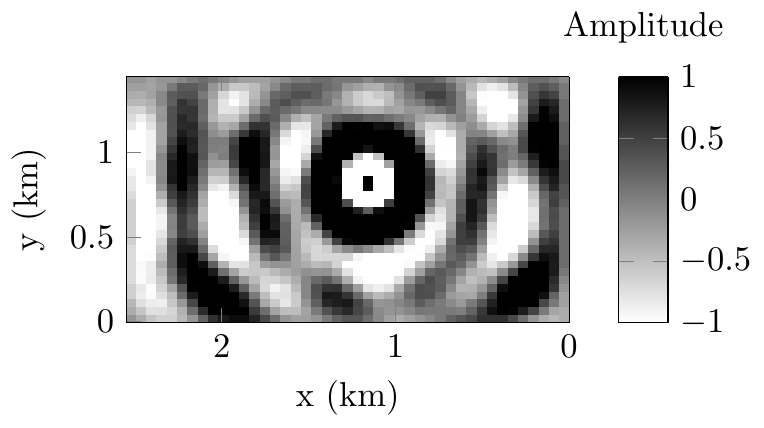}}
\caption{Illustration of the seismic acquisition set.}
\label{fig:acquistion-pml}
\end{figure}

In \cref{fig:stability_depend_n_direct_pml} we compare the 
stability constant estimates using partial data with the 
stability constant estimates obtained when considering the full 
Dirichlet-to-Neumann map as the data. We incorporate the analytical 
lower bound that was computed in the previous test case.\\

\begin{figure}[h!] \centering
 \subfigure[Stability estimates at $5$Hz frequency]
           {\includegraphics[]{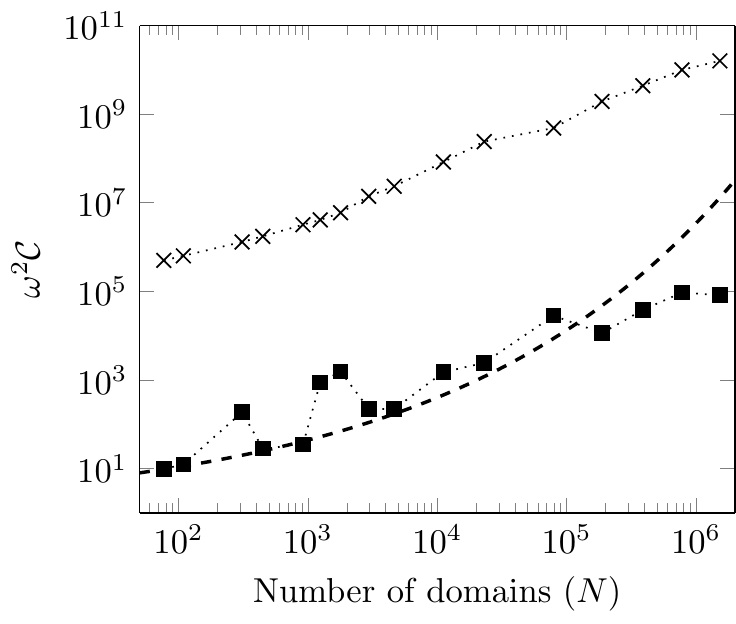}}\hfill
 \subfigure[Stability estimates at $10$Hz frequency]
           {\includegraphics[]{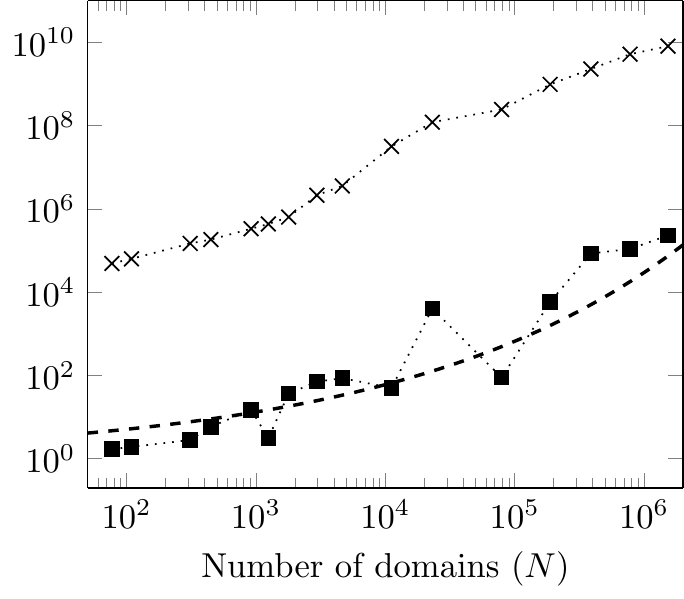}}
 \caption{Comparison of the computational stability estimates 
          using partial data only located on the top boundary
          ($\times$)
          and using the full boundary data
          ($ \blacksquare $).
          The dashed line ($\boldsymbol- \boldsymbol-$)
          represents the analytical lower bound as found 
          in \cref{fig:stability_depend_n_direct}.}
 \label{fig:stability_depend_n_direct_pml}
\end{figure}

The numerical estimates of the stability constants for the full and partial data 
in a $\log \log$-scale differ by a constant. This leads us to our conjecture that 
the $\log\log$ of the stability constants (as a function of $N$) of the full and 
partial data case in the continuous setting differ by a constant.

%% \section{Acknowledgment}
%% We want to thank the referees for their valuable comments and suggestions, 
%% which significantly contributed to improving the quality and clarity of the paper.

%% -------------------------------------------------------------------
%% bibliography ------------------------------------------------------
%% -------------------------------------------------------------------
\bibliographystyle{siamplain}
\bibliography{bibliography}

\end{document}